\documentclass{amsart}
\usepackage{a4wide,amssymb}
\usepackage{hyperref}

\newtheorem{theorem}{Theorem}
\newtheorem{corollary}{Corollary}

\theoremstyle{definition}
\newtheorem{definition}{Definition}

\newcommand{\w}{\omega}
\newcommand{\Ra}{\Rightarrow}

\newcommand{\U}{\mathcal U}
\newcommand{\V}{\mathcal V}

\title[Scattered    compact    sets     in    \v    Cech-complete
spaces]{Scattered  compact  sets  in  continuous  images  of   \v
Cech-complete spaces}

\author{Taras Banakh, Bogdan Bokalo, Vladimir Tkachuk}

\address{T.Banakh: Jan  Kochanowski  University  in Kielce (Poland) and Ivan  Franko  National  University  of  Lviv (Ukraine)}
\email{t.o.banakh@gmail.com}

\address{B.Bokalo:   Ivan  Franko  National  University  of  Lviv
(Ukraine)} \email{b.m.bokalo@gmail.com}

\address{V.Tkachuk:  Universidad Aut\'onoma Metropolitana (Mexico
City, Mexico)} \email{vova@xanum.uam.mx}

\dedicatory{Dedicated  to  the  80-th  birthday   of   A.V.
Arhangel'skii}

\begin{document}
\begin{abstract} Assume that  a  functionally Hausdorff space $X$
is a continuous image of a \v Cech complete space $P$  such  that
$l(P)<\mathfrak   c$.    Then   the   following   conditions  are
equivalent:  (i) every compact  subset  of $X$ is scattered, (ii)
for every continuous map $f:X\to Y$ to a  functionally  Hausdorff
space   $Y$  the  image  $f(X)$  has  cardinality  not  exceeding
$\max\{l(P),\psi(Y)\}$, (iii) no  continuous map $f:X\to[0,1]$ is
surjective.  We also prove the equivalence of the conditions: (a)
$\w_1<\mathfrak b$, and (b) a K-analytic space $X$ 
(with a unique non-isolated point) is countable iff
every  compact  subset  of   $X$   is   countable.
\end{abstract}

\maketitle

A  topological  space  $X$  is {\em scattered} if every non-empty
subspace of $X$ contains an isolated point.

\begin{definition} A topological space $X$ is {\em $k$-scattered}
if every compact subspace of $X$ is scattered. \end{definition}

By \cite[p. 34]{CD}, a \v  Cech-complete space is $k$-scattered if
and only if it is scattered. We recall that a Tychonoff space $X$
is {\em \v Cech-complete} if $X$  is  a  $G_\delta$-set  in  some
(equivalently,     any)     compactification    of    $X$,    see
\cite[3.9.1]{Eng}. A topological space  is  {\em Polish} if it is
homeomorphic  to  a  separable  complete  metric  space.  It   is
well-known  \cite[4.3.26]{Eng}  that  each  Polish  space  is  \v
Cech-complete.

The  following  well-known  characterization of scattered compact
Hausdorff   spaces   can   be    found   in   \cite[\S3]{A2}   or
\cite[8.5.4]{Sem}.

\begin{theorem}\label{t:compact} For a  compact  Hausdorff  space
$X$ the following conditions are equivalent:

\begin{enumerate}

\item $X$ is scattered;

\item for any continuous map $f:X\to Y$ to a metrizable space $Y$
the image $f(X)$ is at most countable;

\item no continuous map $f:X\to[0,1]$ is surjective.
\end{enumerate}

\end{theorem}

In  this  paper we generalize this characterization to K-analytic
spaces, more generally, to functionally Hausdorff spaces $X$ with
analyticity number $\alpha(X)<\mathfrak  c$.   Here $\mathfrak c$
stands for the cardinality of the continuum.

We recall that a  topological  space  $X$  is  {\em  functionally
Hausdorff\/}  if  for any distinct points $x,y\in X$ there exists a
continuous map $f:X\to[0,1]$ such that $f(x)=0$ and $f(y)=1$.

Next, we  recall  the  definitions  of  some cardinal topological
invariants.

The {\em Lindel\"of number} $l(X)$  of a topological space $X$ is
a smallest infinite cardinal $\kappa$ such that  any  open  cover
$\U$  of  $X$ has a subcover $\mathcal V\subset\U$ of cardinality
not exceeding $\kappa$.  A  topological  space $X$ is called {\em
Lindel\"of\/} if $l(X)\le\w$.

The {\em pseudocharacter} $\psi(X)$ of a topological space $X$ is
the  smallest  cardinal $\kappa$ such that for any point $x\in X$
the family $\tau_x$ of all  open  neighborhoods of $x$ contains a
subfamily $\U_x\subset\tau_x$  of  cardinality  $|\U_x|\le\kappa$
such that $\bigcap\U_x=\bigcap\tau_x$.

A topological space $X$ is  defined to be 
\begin{itemize}

\item  {\em  analytic}  if  $X$ is a continuous image of a Polish
space;

\item  {\em  K-analytic}  if  $X$  is  a  continuous  image  of a
Lindel\"of \v Cech-complete space;

\item {\em $\kappa$-analytic} for a cardinal $\kappa$ if $X$ is a
continuous image of a \v Cech-complete space $P$ such that
$l(P)\le\kappa$.

\end{itemize}

Analytic  and K-analytic spaces play an important role in General
Topology \cite{A3},  \cite{Hansell}, 
Descriptive  Set  Theory \cite{Ke}, \cite{RJ}
and Functional Analysis \cite{KKP}.

For a topological space $X$ its {\em  analyticity}  $\alpha(X)$  is
defined as the smallest cardinal $\kappa$ for which the space $X$
is  $\kappa$-analytic.  Since  every  topological  space $X$ is a
continuous image  of  a  discrete  (and  hence  \v Cech-complete)
space, the analyticity is a well-defined  cardinal  invariant  such
that  $$l(X)\le  \alpha(X)\le  |X|+\w.$$  For  a \v Cech-complete
space its analyticity is  equal  to  the Lindel\"of number. Observe
that a topological space $X$  is  $K$-analytic  if  and  only  if
$\alpha(X)=\w$.

\begin{theorem}\label{l:key}  Let  $f:X\to Y$ be a continuous map
from a \v  Cech-complete  space  $X$  to a functionally Hausdorff
space $Y$. If $|f(X)|>\max\{\psi(Y),l(X)\}$, then there exists  a
compact  subset $K\subset X$ whose image $f(K)$ is not scattered.
\end{theorem}

Theorem~\ref{l:key} will be proved in Section~\ref{s:key}. Now we
apply this theorem to prove the following characterization.

\begin{theorem}\label{t:KA} For  a  functionally  Hausdorff space
$X$  with    $\alpha(X)<\mathfrak  c$,   the   following
conditions are equivalent:

\begin{enumerate}

\item $X$ is $k$-scattered;

\item  For  any  continuous  map  $f:X\to  Y$  to  a functionally
Hausdorff  space   $Y$,   the   image   $f(X)$   has  cardinality
not exceeding $\max\{\alpha(X),\psi(Y)\}$.

\item For any continuous map $f:X\to Y$  to  a  metrizable  space
$Y$,  the  image  $f(X)$  has  cardinality  strictly  less   than
$\mathfrak c$.

\item No continuous map $f:X\to[0,1]$ is surjective.

\end{enumerate}

\end{theorem}

\begin{proof}  $(1)\Ra(2)$  Assume that every compact subspace of
$X$ is scattered. To derive a contradiction, assume that for some
continuous  map $f:X\to Y$ to a functionally Hausdorff space $Y$,
the  image   $f(X)$   has  $|f(X)|>\max\{\alpha(X),\psi(Y)\}$.  
By the definition of  the  cardinal
$\alpha(X)$,  there exists a continuous surjective map $g:P\to X$
from  a  \v  Cech-complete   space  $P$  such that
$l(P)=\alpha(X)$.                  Since                 $|f\circ
g(P)|=|f(X)|>\max\{\alpha(X),\psi(Y)\}=\max\{l(P),\psi(Y)\}$,  we
can apply Lemma~\ref{l:key}, and find a compact subset  $K\subset
P$  whose  image $f\circ g(K)$ is not scattered. Then the compact
subset $g(K)$ of $X$ is not scattered, too. \smallskip

The implications $(2)\Ra(3)\Ra(4)$ are trivial.
\smallskip

$(4)\Ra(1)$ Assume that  $X$  contains  a  non-scattered  compact
subset  $K\subset  X$. By Theorem~\ref{t:compact}, there exists a
continuous surjective  map  $f:K\to[0,1]$.  The  space $X$, being
functionally Hausdorff, admits a continuous injective map $g:X\to
Y$ to a compact Hausdorff space $Y$. By the compactness  of  $K$,
the restriction $g{\restriction}K:K\to Y$ is a closed topological
embedding.  By  the Tietze-Urysohn Theorem \cite[2.1.8]{Eng}, the
continuous map  $f\circ (g{\restriction}K)^{-1}:g(K)\to[0,1]$ has
a   continuous   extension   $\varphi:Y\to[0,1]$.   Then    $\bar
f=\varphi\circ  g:X\to[0,1]$ is a continuous extension of the map
$f$. The surjectivity of  $f$  implies  the surjectivity of $\bar
f$.  \end{proof}

Theorem~\ref{t:KA} has the following corollary extending the classical result of Souslin on cardinality of analytic spaces, see \cite[14.13]{Ke}.

\begin{corollary}  Any  functionally  Hausdorff  space  $X$   has
cardinality      
$$\mbox{$|X|\le\max\{\alpha(X),\psi(X)\}$     or
$|X|\ge\mathfrak c$}.$$ 

In  particular, any analytic functionally
Hausdorff space $X$ has  cardinality  $|X|\in\w\cup\{\w,\mathfrak
c\}$. \end{corollary}

For  $F_{\sigma\delta}$-subsets  of compact Hausdorff spaces, the
following corollary of Theorem~\ref{t:KA}  is known, see, e.g., 
\cite[Problem 262]{Tkachuk}.

\begin{corollary}\label{c2} A functionally  Hausdorff  K-analytic
space  $X$  is at most countable if and only if $X$ has countable
pseudocharacter and  every  compact  subset  of  $X$  is  at most
countable. \end{corollary}

Looking at Corollary~\ref{c2} it is  natural  to  ask  \cite{MO2}
whether  the  countability  of the pseudocharacter can be removed
from this characterization.  It turns  out  that this can be done
if  and  only  if  $\w_1<\mathfrak  b$.  We   recall   \cite{vD},
\cite{Vau},  \cite{Blass}  that  $\mathfrak  b$  is  the smallest
cardinality of a subset $B\subset\w^\w$ such that for every $y\in
\w^\w$ there exists  $x\in  B$  such  that  $x\not\le^* y$, where
$x\le^* y$ means that the set  $\{n\in\w:x(n)\not\le  y(n)\}$  is
finite.   It  is  well-known  \cite{vD}, \cite{Vau}, \cite{Blass}
that $\w_1<\mathfrak  b=\mathfrak  c$  under  MA$+\neg$CH. On the
other hand, $\w_1=\mathfrak b=\mathfrak c$ under CH.

The equivalence $(1)\Leftrightarrow(2)$ in the following  theorem
was proved by Fremlin \cite{Fremlin}.

\begin{theorem} The following conditions are equivalent:
\begin{enumerate} 
\item $\w_1<\mathfrak b$.
\item A K-analytic Hausdorff space $X$ is analytic iff  every compact subset of $X$ is metrizable.
\item A K-analytic space $X$ is countable iff every compact subset of $X$ is countable.
\item A K-analytic Hausdorff space $X$ with a unique non-isolated point is countable iff every compact subset of $X$ is countable.
\end{enumerate}
\end{theorem}

\begin{proof} Since  $(1)\Leftrightarrow(2)$  
was  established by
Fremlin    \cite{Fremlin},    it    remains    to    prove   that
$(1)\Ra(3)\Ra(4)\Ra(1)$. The implication $(3)\Ra(4)$ is trivial.

\smallskip
$(1)\Ra(3)$.   Assume   that  $\w_1<\mathfrak  b$  and  take  any
K-analytic space $X$ such  that  every  compact  subset of $X$ is
countable. By definition, the K-analytic space $X$ is  the  image
of  a  Lindel\"of  \v  Cech-complete space $P$ under a continuous
surjective map $f:P\to X$. Let  $\bar P$ be a compactification of
$P$ and let $(W_n)_{n\in\w}$ be a  decreasing  sequence  of  open
sets in $\bar P$ such that $P=\bigcap_{n\in\w}W_n$.  Since $X$ is
Lindel\"of,   we   can   assume  that  each  open  set  $W_n$  is
$\sigma$-compact.   So,   $W_n=\bigcup_{m\in\w}K_{n,m}$   for  an
increasing sequence $(K_{n,m})_{m\in\w}$ of  compact  subsets  of
$\bar  P$.   For every infinite sequence $s\in\w^\w$ consider the
compact   set  $K_s=\bigcap_{n\in\w}K_{n,s{\restriction}n}\subset
\bigcap_{n\in\w}W_n=P$  and  observe  that  $X=f(P)=\bigcup_{s\in
\w^\w}f(K_s)$, and $K_s\subset K_t$ for any sequences $s\le t$ in
$\w^\w$. By our assumption, all compact sets in $X$ are  at  most
countable.   Consequently, for every $s\in \w^\w$ the compact set
$f(K_s)$ is at most countable.

Assuming that $X$  is  uncountable,  we can construct transfinite
sequences of points $\{s_\alpha\}_{\alpha<\w_1}\subset \w^\w$ and
$\{x_\alpha\}_{\alpha\in\w^\w}\subset X$ such  that  $x_\alpha\in
f(K_{s_\alpha})\setminus    \bigcup_{\beta<\alpha}f(K_{s_\beta})$
for every $\alpha<\w_1$.

The  definition  of  the cardinal $\mathfrak b>\w_1$ implies that
there exists a countable subset $T\subset\w^\w$ such that for any
$\alpha<\w_1$ there exists $t\in T$  with $s_\alpha\le t$. By the
Pigeonhole   Principle,   for   some    $t\in    T$    the    set
$\Omega=\{\alpha<\w_1:s_\alpha\le  t\}$  is uncountable. Then the
countable   set   $f(K_t)$    contains    the   uncountable   set
$\bigcup_{\alpha\in\Omega}f(K_{x_\alpha})\supset
\{x_\alpha:\alpha\in\Omega\}$, which is a desired  contradiction,
showing that $X$ is at most countable. \smallskip

$(4)\Ra(1)$  Assuming  that  $\w_1=\mathfrak  b$,  we  can find a
transfinite  sequence  $\{x_\alpha\}_{\alpha\in\w_1}\subset\w^\w$
such that for every $y\in\w^\w$ there exists $\alpha\in\w_1$ such
that $x_\alpha\not\le^*  y$.  Using  the  definition of
$\mathfrak  b$,  for  every  $\alpha<\w_1$  choose   a   function
$y_\alpha\in\w^\w$   such   that   $x_\alpha\le   y_\alpha$   and
$x_\beta\le  y_\beta\le^*  y_\alpha$ for all $\beta<\alpha$. Then
for every $y\in\w^\w$  the  set $\{\alpha\in\w_1:y_\alpha\le y\}$
is  countable,  which  implies  that  for  every  compact  subset
$K\subset\w^\w$                 the                  intersection
$K\cap\{y_\alpha\}_{\alpha\in\w_1}$ is countable.

Let  $Y=\{y_\alpha\}_{\alpha\in\w_1}$  and  $X=\{\infty\}\cup  Y$
where $\infty\notin Y$ is any point. Endow $X$ with the Hausdorff
topology     generated    by    the    base    $$\big\{\{y\}:y\in
Y\big\}\cup\{X\setminus D:D\subset Y\mbox{ is closed and discrete
in }\w^\w\}$$and observe that  $\infty$  is a unique non-isolated
point of $X$, which implies that the space $X$ is  (hereditarily)
normal.
  
Consider   the   compact-valued   function  
$$\Phi:\w^\w\multimap
X,\;\;\Phi:x\mapsto \{\infty,x\}\cap X. $$ 
and observe that it is
upper semi-continuous  in  the  sense  that  for  every  open set
$U\subset X$ the set $\{x\in \w^\w:\Phi(x)\subset U\}$ is open in
$\w^\w$. Since $X=\bigcup_{x\in\w^\w}\Phi(x)$, the space  $X$  is
K-analytic, see \cite[3.1]{Hansell}.

We  claim that every compact subset $K\subset X$ is countable. To
derive a  contradiction,  assume  that  $K$  is uncountable. Then
$K\cap Y$ is uncountable and the closure $\overline{K\cap Y}$  of
$K\cap  Y$  in $\w^\w$ is not compact (as the intersection of $Y$
with  any  compact  subset   of   $\w^\w$  is  countable).  Since
$\overline{K\cap Y}$ is not compact,  there  exists  an  infinite
subset  $D\subset  K\cap  Y$,  which  is  closed  and discrete in
$\w^\w$. The definition of the  topology  of $X$ ensures that the infinite 
set $D\subset K$ is closed and discrete  in  $X$,  which  is  not
possible as $K$ is compact.  \end{proof}

\section{Proof of Theorem~\ref{l:key}}\label{s:key}

Let  $f:X\to  Y$  be  a  continuous  surjective  map  from  a  \v
Cech-complete  space  $X$ into a functionally Hausdorff space $Y$
and  $\kappa=\max\{l(X),\psi(Y)\}$.  The   space  $X$,  being  \v
Cech-complete, is a $G_\delta$-set in some compactification $\bar
X$  of   $X$.   Then   there   exists   a   decreasing   sequence
$(W_n)_{n\in\w}$   of   open   sets   in   $\bar   X$  such  that
$X=\bigcap_{n\in\w}W_n$ and $W_0=\bar X$.

Let $2=\{0,1\}$ and  $2^{<\w}=\bigcup_{n\in\w}2^n$  be the family
of  all  finite  binary  sequences.   For   a   binary   sequence
$s=(s_0,\dots,s_{n-1})\in  2^{<\w}$ and a number $i\in\{0,1\}$ by
$s\hat{\;}i$ we denote the sequence $(s_0,\dots,s_{n-1},i)$.

Assuming that $|f(X)|>\kappa$,  we  shall inductively construct a
family $(U_s)_{s\in 2^{<\w}}$ of open  $\sigma$-compact  sets  in
$\bar  X$  such  that  for  every  $n\in\w$  and $s\in 2^{n}$ the
following conditions are satisfied:

\begin{itemize}

\item[$(a_s)$] $|f(X\cap U_s)|>\kappa$;

\item[$(b_s)$]
$\overline{U}_{s\hat{\;}0}\cup\overline{U}_{s\hat{\;}1}\subset
U_s\subset W_n$;

\item[$(c_s)$]
$\overline{U}_{s\hat{\;}0}\cap\overline{U}_{s\hat{\;}1}=\emptyset$;

\item[$(d_s)$]  $f(X\cap  \overline{U}_{s\hat{\;}0})\cap f(X\cap
\overline{U}_{s\hat{\;}1})=\emptyset$.

\end{itemize}

We start the inductive  construction letting $U_\emptyset=\bar X$
for the unique element $\emptyset\in 2^0$. Assume that  for  some
$n\in\w$   and   $s\in   2^n$   we   have   constructed  an  open
$\sigma$-compact set $U_s\subset\bar  X$ satisfying the condition
$(a_s)$. We shall construct  two  open  $\sigma$-compact  subsets
$U_{s\hat{\;}0},U_{s\hat{\;}1}$    in    $U_s$   satisfying   the
conditions     $(b_s),(c_s),(d_s)$     and    $(a_{s\hat{\;}0})$,
$(a_{s\hat{\;}1})$.  

Let $Y_{s}\subset Y$ be the set  of points $y\in Y$ possessing an
open neighborhood $O_y\subset  Y$  such  that  the  set  $O_y\cap
f(X\cap U_s)$ has cardinality $\le \kappa$. We claim that the set
$f(X)\setminus  Y_s$  is  infinite.  To  derive  a contradiction,
assume that the  set  $f(X)\setminus  Y_s$  is  finite. Since the
space  $Y$  is  Hausdorff with $\psi(Y)\le\kappa$, there exists a
family $\mathcal V$ of open sets in $Y$ such that $|\V|\le\kappa$
and  $\bigcap\V=f(X)\setminus  Y_s$.    Observe  that  for  every
$V\in\V$, for the closed subspace $f(X)\setminus V$ of the  space
$f(X)$ we have $l(f(X)\setminus V)\le l(f(X))\le l(X)\le \kappa$.
Then  there  exists  a  set  $Y_V\subset  f(X)\setminus  V\subset
Y_s\setminus V$ of cardinality $|Y_V|\le l(f(X)\setminus V)$ such
that   $f(X)\setminus   V\subset\bigcup_{y\in   Y_V}O_y$.    Then
$$\kappa<|f(X\cap   U_s)|=|f(X\cap   U_s)\setminus  Y_s|+|f(X\cap
U_s)\cap    Y_s|\le|f(X)\setminus    Y_s|+\sum_{V\in\V}\sum_{y\in
Y_V}|O_y\cap  f(X\cap  U_s)|\le\kappa,$$   which   is  a  desired
contradiction  proving  that  the  set  $f(X)\setminus  Y_s$   is
infinite.

Since  $Y$  is functionally Hausdorff, we can choose two distinct
points $y_0,y_1\in f(X)\setminus S$  and two open $F_\sigma$-sets
$V_0,V_1\subset             Y$             such              that
$\overline{V}_0\cap\overline{V}_1=\emptyset$ and $y_i\in V_i$ for
$i\in\{0,1\}$.   For   every   $i\in\{0,1\}$,  the  non-inclusion
$y_i\notin Y_s$ ensures that  $|V_i\cap f(X\cap U_s)|>\kappa$. It
follows  that   $V_{s,i}:=U_s\cap   f^{-1}(V_i)$   is   an   open
$F_\sigma$-set  in the open $F_\sigma$-subset $X\cap U_s$ of $X$.
Consequently, $l(V_{s,i})\le l(X)\le\kappa$.  Find an open subset
$W_{s,i}\subset U_s\cap W_{n}$ such that $V_{s,i}=X\cap W_{s,i}$.
Consider the closure $\overline{W}_{s,i}$ of $W_{s,i}$  in  $\bar
X$    and    observe    that   $f(X\cap\overline{W}_{s,i})\subset
\overline{V}_{i}$,          which           implies          that
$\overline{W}_{s,0}\cap\overline{W}_{s,1}\cap       X=\emptyset$.
Replacing    the    sets    $W_{s,0}$    and     $W_{s,1}$     by
$W_{s,0}\setminus\overline{W}_{s,1}$                          and
$W_{s,1}\setminus\overline{W}_{s,0}$, respectively, we can assume
that
$$W_{s,0}\cap\overline{W}_{s,1}=
\emptyset=W_{s,1}\cap\overline{W}_{s,0}.$$

Since  the  space  $\bar  X$  is compact and Hausdorff, for every
point $x\in  V_{s,i}$  we  can  choose  an  open $\sigma$-compact
neighborhood  $O_x$  in  $\bar  X$  such  that  $\bar  O_x\subset
W_{s,i}$.  Since  $l(V_{s,i})\le  \kappa$,  there  exists  a  set
$X_{s,i}\subset V_{s,i}$ of cardinality $|X_{s,i}|\le\kappa$ such
that   $V_{s,i}\subset\bigcup_{x\in   X_{s,i}}O_x$   and    hence
$V_{s,i}=\bigcup_{x\in  X_{s,i}}X\cap O_x$. Then $V_i\cap f(X\cap
U_s)=f(V_{s,i})=\bigcup_{x\in   X_{s,i}}f(X\cap    O_x)$.   Since
$|V_i\cap f(X\cap U_s)|>\kappa\ge|X_{s,i}|$, there exists a point
$x_{s,i}\in X_{s,i}$ such that  $|f(X\cap  O_{x_{s,i}})|>\kappa$.
Put  $U_{s\hat{\;}i}:=O_{x_{s,i}}$,  and  observe  that  the sets
$U_{s\hat{\;}0}$  and  $U_{s\hat{\;}1}$  satisfy  the  conditions
$(b_s),(c_s),(d_s)$ and  $(a_{s\hat{\;}0})$,  $(a_{s\hat{\;}1})$.
\smallskip

After completing the inductive construction, consider the compact
set    $K=\bigcap_{n\in\w}\bigcup_{s\in    2^n}\bar    U_s\subset
\bigcap_{n\in\w}W_{n+1}=X$.  It  remains  to prove that the image
$f(K)\subset Y$ is not scattered.

The conditions $(b_s)$,  $(c_s)$,  of  the inductive construction
imply that for  every  point  $x\in  K$  there  exists  a  unique
sequence          $s_x\in         2^\w$         such         that
$x \in \bigcap_{n\in\w} \overline{U}_{s_x{\restriction}n}= 
\bigcap_{n\in\w}U_{s{\restriction}n}$.
It follows that  the  map  $\varphi:K\to 2^\w$, $\varphi:x\mapsto
s_x$, is continuous and surjective. We claim that for any $x,y\in
K$ with $\varphi(x)\ne\varphi(y)$ the points  $f(x)$  and  $f(y)$
are  distinct.   Since  $s_x=\varphi(x)\ne\varphi(y)=s_y$,  there
exists        a       unique       $n\in\w$       such       that
$s_x{\restriction}n=s_y{\restriction}n$  but  $s_x(n)\ne s_y(n)$.
The        inductive        condition         $(d_t)$         for
$t=s_x{\restriction}n=s_y{\restriction}n$       ensures      that
$\{f(x)\}\cap\{f(y)\}\subset  f(X\cap   \bar  U_{t\hat{\;}0})\cap
f(X\cap  \bar  U_{t\hat{\;}1})=\emptyset$,   which   means   that
$f(x)\ne f(y)$. Then there exists a unique function $\phi:f(K)\to
2^\w$ such that $\varphi=\phi\circ f{\restriction}K$. Taking into
account  that  the  map  $\varphi$  is  continuous  and  the  map
$f{\restriction}K:K\to  f(K)$  is closed (and hence quotient), we
conclude that the  map  $\phi:f(K)\to  2^\w$  is continuous.  So,
$f(K)$ admits a continuous map onto $2^\w$  and  hence  admits  a
surjective      continuous      map     onto     $[0,1]$.      By
Theorem~\ref{t:compact},  the   compact   space   $f(K)$  is  not
scattered.

\section{Acknowledgement}
 
The  first  author  would  like  to  express  his  thanks to Saak
Gabriyelyan  whose  stimulating  questions  lead  to  the results
presented in this paper. These results are exploited in the joint
paper \cite{BG}.

\end{document}